\documentclass[11pt,a4paper]{amsart}
\usepackage{amsmath,amscd,upref}
\usepackage[psamsfonts]{amssymb}
\usepackage[all]{xy}

\newtheorem{thm}{Theorem}[section]
\newtheorem{crl}[thm]{Corollary}
\newtheorem{lmm}[thm]{Lemma}
\newtheorem{prp}[thm]{Proposition}
\theoremstyle{definition}
\newtheorem{dfn}[thm]{Definition}
\newtheorem{exa}[thm]{Example}
\theoremstyle{remark}
\newtheorem*{rem}{Remark}
\numberwithin{equation}{section}

\DeclareMathOperator{\Top}{\bf Top}

\DeclareMathOperator{\Diff}{\bf Diff}
\DeclareMathOperator{\NumGen}{\bf NG}
\DeclareMathOperator{\FCW}{\bf FCW}
\DeclareMathOperator{\map}{\bf map}
\DeclareMathOperator{\smap}{\bf smap}
\DeclareMathOperator{\C}{\bf C}
\DeclareMathOperator{\SLEF}{\bf SLEF}
\DeclareMathOperator{\Spec}{\bf Spec}

\DeclareMathOperator{\op}{op}
\DeclareMathOperator{\ev}{ev}

\DeclareMathOperator*{\colim}{colim}

\newcommand{\Cinfty}{C^{\infty}}
\newcommand{\spec}[1]{\partial{#1}}

\title{Homology and cohomology via enriched bifunctors}
\author[K. Shimakawa]{Kazuhisa Shimakawa}
\address{Graduate School of Natural Science and Technology,
  Okayama University, Okayama 700-8530 Japan}
\author{Kohei Yoshida}
\address{Graduate School of Natural Science and Technology,
  Okayama University, Okayama 700-8530 Japan}
\author{Tadayuki Haraguchi}
\address{Graduate School of Natural Science and Technology,
  Okayama University, Okayama 700-8530 Japan}
\dedicatory{To the memory of Professor Nobuo Shimada}
\begin{document}
\begin{abstract}
  Let $\Top$ and $\Diff$ be the categories of topological and
  diffeological spaces, respectively.
  By using an adjunction between $\Top$ and $\Diff$ we show that
  the full subcategory $\NumGen$ of $\Top$ consisting of numerically
  generated spaces is complete, cocomplete, and cartesian closed.
  In fact, $\NumGen$ can be embedded into $\Diff$ as a cartesian closed
  full subcategory.
  It follows then that the category $\NumGen_0$ of numerically generated
  pointed spaces is complete, cocomplete, and monoidally closed
  with respect to the smash product.
  These features of $\NumGen_0$ are used to establish a simple but
  flexible method for constructing generalized homology and
  cohomology theories by using the notion of enriched bifunctors.
\end{abstract}

\maketitle
\section{Introduction}
\label{sec:introduction}

On the category of pointed topological spaces,
generalized homology and cohomology theories are usually constructed
by using spectra.
However, the definition of a spectrum is not suitable
for more general type of model categories,
e.g.\ the category of equivariant spaces,
because it is too tightly connected to the framework of
classical stable homotopy theory.
In this paper, we present an alternative and more category theoretical
approach to homology and cohomology theories which is based on
the notion of a linear enriched functor instead of a spectrum.

Our strategy is as follows:
We first replace the category $\Top_0$ of pointed topological spaces by
a convenient full subcategory $\NumGen_0$ which fulfills our requirements.
More explicitly, we call a topological space $X$ numerically generated
if it has the final topology with respect to its singular simplexes,
and define $\NumGen_0$ to be the full subcategory of $\Top_0$
consisting of numerically generated spaces.
Then $\NumGen_0$ is complete, cocomplete, and is monoidally closed
in the sense that there is an internal hom $Z^Y$ 
satisfying a natural bijection
\[
\hom_{\NumGen_0}(X \wedge Y,Z) \cong \hom_{\NumGen_0}(X,Z^Y).
\]
Moreover, there exist a reflector $\nu \colon \Top_0 \to \NumGen_0$
such that the natural map $\nu{X} \to X$ is a weak equivalence,
and a sequence of weak equivalences
\[
Y^X \leftarrow \nu\map_0(X,Y) \rightarrow \map_0(X,Y)
\]
where $\map_0(X,Y)$ is the set of pointed maps from $X$ to $Y$
equipped with the compact-open topology.
Thus $\NumGen_0$ is eligible, from the viewpoint of homotopy theory,
as a convenient replacement for $\Top_0$.

A functor $T \colon \NumGen_0 \to \NumGen_0$ is called enriched if
the correspondence which maps $f \colon X \to Y$ to $Tf \colon TX \to TY$
induces a morphism $Y^X \to TY^{TX}$ between internal homs,
and is called linear if it converts a cofibration sequence
into a homotopy fibration sequence.
%
Theorem~\ref{thm:homology_theory} states that every linear enriched functor
$T$ defines a generalized homology theory $X \mapsto \{ h_n(X;T) \}$
such that $h_n(X;T)$ is isomorphic to $\pi_n TX$ for $n \geq 0$,
and to $\pi_0 T(\Sigma^{-n}X)$ for $n < 0$.
We see that every homology theory represented by a spectrum comes from
a linear enriched functor, and vice versa.
In fact, the functor which assigns to a linear enriched functor
its derivative, in the sense of Goodwillie~\cite{goodwillie_I},
at $S^0$ induces a Quillen equivalence between the model categories of
linear enriched functors and spectra.
Thus the homotopy category of linear enriched functors is
equivalent to the stable category.
A benefit of our approach is its flexibility:
For example, the enrichedness of $T$ implies the orthogonality,
hence the symmetricity, of its derivatives
$\{ T(\Sigma^n X) \mid n \geq 0 \}$.

Our method for constructing cohomology theories is based on
the notion of a bifunctor
\(
\NumGen_0^{\op} \times \NumGen_0 \to \NumGen_0,
\)
contravariant with respect to the first argument and covariant
with respect to the second argument.
There are bivariant version of the notions of enrichedness and linearity,
and Theorem~\ref{thm:bivariant_theory} states that every bifunctor $F$
which is bivariantly enriched and linear determines a pair of
a generalized homology theory $X \mapsto \{h_n(X;F)\}$ and
a cohomology theory $X \mapsto \{h^n(X;F)\}$ such that
\[
h_n(X;F) \cong \pi_0 F(S^{n+k},\Sigma^k X), \quad
h^n(X;F) \cong \pi_0 F(\Sigma^k X,S^{n+k})
\]
hold whenever $k$ and $n+k$ are non-negative.
Just as a spectrum represents both homology and cohomology theories,
every linear enriched functor $T \colon \NumGen_0 \to \NumGen_0$
produces homology and cohomology theories through the bilinear functor
$F$ given by the formula $F(X,Y) = (TY)^X$. 
However, there exist enriched bifunctors not of the form as above
but induce interesting bivariant homology-cohomology theories.
(Cf.\ Remark after Theorem~\ref{thm:bivariant_theory}.)

Since our approach is based solely on the standard category theoretical
notions such as enrichedness and linearity,
it can be easily extended to other model categories consisting of
simplicial sets, diffeological spaces, equivariant
(topological or diffeological) spaces, etc.
Some of such extensions will be discussed in a subsequent paper.
\section{Diffeological spaces}
\label{sec:relation}

To derive the properties of the category of numerically generated spaces
we use an adjunction between the categories of topological and
diffeological spaces, respectively.
Recall from \cite{zemmour:diffeology} that
a diffeological space consists of a set $X$ together with
a family $D$ of maps from open subsets of Euclidean spaces into $X$
satisfying the following conditions:

\smallskip

\paragraph{\bf Covering} Any constant parametrization $\mathbf{R}^n \to X$
belongs to $D$.

\paragraph{\bf Locality} A parametrization $\sigma \colon U \to X$ belongs
to $D$ if every point $u$ of $U$ has a neighborhood $W$ such that
$\sigma|W \colon W \to X$ belongs to $D$.

\paragraph{\bf Smooth compatibility} If $\sigma \colon U \to X$ belongs to
$D$, then so does the composite $\sigma f \colon V \to X$ for any
smooth map $f \colon V \to U$ between open subsets of Euclidean
spaces.

\smallskip

We call $D$ a diffeology of $X$, and each member of $D$ a plot of $X$.

A map between diffeological spaces $f \colon X \to Y$ is called smooth
if for any plot $\sigma \colon U \to X$ of $X$ the composite
$f\sigma \colon U \to Y$ is a plot of $Y$.
In particular, if $D$ and $D'$ are diffeologies on a set $X$
then the identity map $(X,D) \to (X,D')$ is smooth if and only if
$D$ is contained in $D'$.
In that case, we say that $D$ is finer than $D'$, or
$D'$ is coarser than $D$.
Clearly, the class of Diffeological spaces and smooth maps form
a category $\Diff$. 

\begin{thm}
  The category $\Diff$ is complete, cocomplete, and cartesian closed.
\end{thm}

A category is complete if it has equalizers and small products,
and is cocomplete if it has coequalizers and small coproducts.
Therefore, the theorem follows from the basic constructions given below.

\smallskip

\paragraph{\bf Products}
Given diffeological spaces $X_j$, $j \in J$, their product
is given by a pair $(\prod_{j \in J}X_j,D)$, where $D$ is the set of
parametrizations $\sigma \colon U \to \prod_{j \in J}X_j$ such that
every its component $\sigma_j \colon U \to X_j$ is a plot of $X_j$.

\paragraph{\bf Coproducts}
The coproduct of $X_j$, $j \in J$, is given by
$(\coprod_{j \in J}X_j,D)$, where $D$ is the set of parametrizations
$\sigma \colon U \to \coprod_{j \in J}X_j$ which can be written locally
as the composite of the inclusion $X_j \to \coprod_{j \in J}X_j$
with a plot of $X_j$.  

\paragraph{\bf Subspaces}
Any subset $A$ of a diffeological space $X$ is itself a diffeological space
with plots given by those parametrizations $\sigma \colon U \to A$
such that the post composition with the inclusion $A \to X$ is a plot of $X$.

\paragraph{\bf Quotients}
Let $p \colon X \to Y$ be a surjection from a diffeological space $X$
to a set $Y$.
Then $Y$ inherits from $X$ a diffeology consisting of those parametrizations
$\sigma \colon U \to Y$
which lifts locally, at every point $u \in U$, along $p$.

\paragraph{\bf Exponentials}
Given diffeological spaces $X$ and $Y$, the set $\hom_{\Diff}(X,Y)$
has a diffeology $D_{X,Y}$ consisting of those
$\sigma \colon U \to \hom_{\Diff}(X,Y)$ such that
for every plot $\tau \colon V \to X$ of $X$, the composite
\[
U \times V \xrightarrow{(\sigma,\tau)} \hom_{\Diff}(X,Y) \times X
\xrightarrow{\ev} Y
\]
is a plot of $Y$.
Putting it differently, $D_{X,Y}$ is the coarsest diffeology such that
the evaluation map $\hom_{\Diff}(X,Y) \times X \to Y$ is smooth.

\smallskip

Let us denote by $\Cinfty(X,Y)$ the diffeological space
$(\hom_{\Diff}(X,Y),D_{X,Y})$.
Then there is a natural map
$\alpha \colon \Cinfty(X \times Y,Z) \to \Cinfty(X,\Cinfty(Y,Z))$
given by the formula: $\alpha(f)(x)(y) = f(x,y)$ for $x \in X$ and
$y \in Y$.
The following exponential law implies the cartesian closedness of $\Diff$.

\begin{thm}[{\cite[1.60]{zemmour:diffeology}}]
  \label{thm:closedness_for_Diff}
  The map $\alpha$ induces a smooth isomorphism
  \[
  \Cinfty(X \times Y,Z) \cong \Cinfty(X,\Cinfty(Y,Z)).
  \]
\end{thm}
\section{Numerically generated spaces}
\label{sec:limits}

Given a topological space $X$, let $DX$ be the diffeological space with
the same underlying set as $X$ and with all continuous maps
from open subsets of Euclidean spaces into $X$ as plots.
Clearly, a continuous map $f \colon X \to Y$ induces a smooth map $DX \to DY$.
Hence there is a functor $D \colon \Top \to \Diff$ which maps
a topological space $X$ to the diffeological space $DX$.

On the contrary, any diffeological space $X$ determines a topological space
$TX$ having the same underlying set as $X$ and is equipped with
the final topology with respect to the plots of $X$.
Any smooth map $f \colon X \to Y$ induces a continuous map $TX \to TY$,
hence we have a functor $T \colon \Diff \to \Top$.

\begin{prp}
  The functor $T$ is a left adjoint to $D$.
\end{prp}

\begin{proof}
  Let $X$ be a diffeological space and $Y$ a topological space.
  Then a map $f \colon TX \to Y$ is continuous if and only if
  the composite $f\circ \sigma$ is continuous for every plot $\sigma$ of $X$.
  But this is equivalent to say that $f \colon X \to DY$ is smooth.
  Thus the natural map
  \[
  \hom_{\Top}(TX,Y) \to \hom_{\Diff}(X,DY)
  \]
  is bijective for every $X \in \Diff$ and $Y \in \Top$.
\end{proof}

\begin{prp}
  A topological space $X$ is numerically generated if and only if
  the counit of the adjunction $TDX \to X$ is a homeomorphism.
\end{prp}

\begin{proof}
  The condition $TDX = X$ holds if and only if
  $X$ has the final topology with respect to all the continuous maps
  from an open subset of a Euclidean space into $X$.
  But this is equivalent to say that $X$ has the final topology
  with respect to the singular simplexes of $X$.
\end{proof}

Let us write $\nu = TD$, so that $X$ is numerically generated
if and only if $\nu{X} = X$ holds.

\begin{lmm}
  For any topological space $X$ we have $\nu(\nu{X}) = \nu{X}$.
\end{lmm}

\begin{proof}
  Every plot $\sigma \colon U \to X$ of $X$ lifts to a plot of $\nu{X}$,
  since $\nu{U} = U$ holds for any open subset $U$ of $\mathbf{R}^n$.
  Thus $\nu{X}$ has the same plots as $X$, and hence $\nu(\nu{X})$
  has the same topology as $\nu{X}$.
\end{proof}

It follows that $\NumGen$ is reflective in $\Top$,
and the correspondence $X \mapsto \nu{X}$ induces a reflector
$\nu \colon \Top \to \NumGen$.

\begin{prp}
  The category $\NumGen$ is complete and cocomplete.
  For every small diagram $F \colon J \to \NumGen$, we have
  \begin{eqnarray*}
    & \lim_J F \cong T(\lim_J DF) \cong \nu(\lim_J IF) \\
    & \colim_J F \cong T(\colim_J DF) \cong \colim_J IF
  \end{eqnarray*}
  where $I$ denotes the inclusion functor $\NumGen \to \Top$.
\end{prp}

\begin{proof}
  Since $\Diff$ is complete, the diagram $DF \colon J \to \Diff$
  has a limiting cone $\{ \phi_j \colon \lim_J DF \to DF(j) \}$.
  We shall show that the cone
  \[
  \{ T\phi_j \colon T({\lim}_J DF) \to TDF(j) = F(j) \}
  \]
  is a limiting cone to $F$.
  Let $\{ \psi_j \colon X \to F(j) \}$ be an arbitrary cone to $F$.
  Then $\{ D\psi_j \colon DX \to DF(j) \}$ is a cone to $DF$,
  hence there is a unique morphism $u \colon DX \to \lim_J DF$
  such that $D\psi_j = \phi_j\circ u$ holds.
  But then
  \( Tu \colon X = TDX \to T(\lim_J DF) \)
  is a unique morphism such that $\psi_j = T\phi_j \circ Tu$ holds.
  Hence $\{ T\phi_j \}$ is a limiting cone to $F$.
  Since the right adjoint functor $D$ preserves limits, we have
  \( T(\lim_J DF) \cong TD(\lim_J IF) = \nu(\lim_J IF). \)

  Similar argument shows that $T(\colim_J DF)$ is a colimit of $F$.
  But in this case we have
  $T(\colim_J DF) \cong \colim_J \nu F = \colim_J IF$,
  because the left adjoint functor $T$ preserves colimits.
\end{proof}
\section{Exponentials in $\NumGen$}

A map $f \colon X \to Y$ between topological spaces is said to be
numerically continuous if the composite
$f\circ \sigma \colon \Delta^n \to Y$ is continuous
for every singular simplex $\sigma \colon \Delta^n \to X$.
Clearly, we have the following.

\begin{prp}
  \label{prp:numerical continuity}
  Let $f \colon X \to Y$ be a map between topological spaces.
  Then the following conditions are equivalent:
  \begin{enumerate}
  \item $f \colon X \to Y$ is numerically continuous.
  \item $f\circ \sigma \colon U \to Y$ is continuous
    for any continuous map $\sigma \colon U \to X$
    from an open subset $U$ of a Euclidean space into $X$.
  \item $f \colon \nu{X} \to Y$ is continuous.
  \item $f \colon DX \to DY$ is smooth.
  \end{enumerate}
\end{prp}

Let us denote by $\map(X,Y)$ the set of continuous maps from $X$ to $Y$
equipped with the compact-open topology,
and let $\smap(X,Y)$ be the set of numerically continuous maps
equipped with the initial topology with respect to the maps
\[
\sigma^* \colon \smap(X,Y) \to \map(\Delta^n,Y),\quad
\sigma^*(f) = f\circ \sigma,
\]
where $\sigma \colon \Delta^n \to X$ runs through singular simplexes of $X$.
More explicitly, the space $\map(X,Y)$ has a subbase consisting of
those subsets
\[
W(K,U) = \{ f \mid f(K) \subset U \},
\]
where $K$ is a compact subset of $X$ and $U$ is an open subset of $Y$.
On the other hand, $\smap(X,Y)$ has a subbase consisting of those subsets
\[
W(\sigma,L,U) = \{ f \mid f(\sigma(L)) \subset U \},
\]
where $\sigma \colon \Delta^n \to X$ is a singular simplex,
$L$ a compact subset of $\Delta^n$, and $U$ an open subset of $Y$.
As we have
\[
W(\sigma,L,U) \cap \map(X,Y) = W(\sigma(L),U),
\]
the inclusion map $\map(X,Y) \to \smap(X,Y)$ is continuous.

\begin{prp}
  \label{prp:smap is map for numerically generated space}
  The inclusion map $\map(X,Y) \to \smap(X,Y)$ is bijective
  for all $Y$ if, and only if, $X$ is numerically generated.
\end{prp}

\begin{proof}
  If $X$ is numerically generated then a numerically continuous map
  $f \colon X \to Y$ is automatically continuous.
  Hence $\map(X,Y) \to \smap(X,Y)$ is surjective for all $Y$.
  Conversely, if $\map(X,\nu{X}) \to \smap(X,\nu{X})$ is surjective,
  then the unit of the adjunction $X \to \nu{X}$ is continuous,
  implying that $X$ is numerically generated.
\end{proof}

\begin{prp}
  \label{prp:smap is map for CW-complexes}
  Suppose that $X$ is a CW-complex.  Then the inclusion map
  $\map(X,Y) \to \smap(X,Y)$ is a homeomorphism for any $Y$.
\end{prp}

\begin{proof}
  Since $X$ has the weak topology with respect to the family of closed cells,
  a map $f$ from $X$ to $Y$ is continuous if and only if
  it is numerically continuous.
  Hence 
  $\map(X,Y) \to \smap(X,Y)$ is a bijection.
  To prove the continuity of its inverse, we have to show
  that every subset of the form $W(K,U)$ is open in $\smap(X,Y)$.
  Since $X$ is closure-finite, $K$ is contained in a finite complex, say $A$.
  Let $\{ e_1,\dots,e_k \}$ be the set of cells of $A$,
  and let $L_i = \psi_i^{-1}(\overline{e_i} \cap K) \subset \Delta^{n_i}$,
  where $\psi_i \colon \Delta^{n_i} \to X$ is a characteristic map for $e_i$.
  Then we have
  \[
  W(K,U) = W(\psi_1,L_1,U) \cap \cdots \cap W(\psi_k,L_k,U).
  \]
  Hence $W(K,U)$ is open in $\smap(X,Y)$.
\end{proof}

The proposition above implies that any CW-complex $X$ is numerically generated.
Thus we have the following.

\begin{crl}
  The category $\NumGen$ contains all CW-complexes.
\end{crl}

\begin{rem}
  If $X$ and $Y$ are CW-complexes then so is the product $X \times Y$ in
  $\NumGen$.  
  In fact, if $\{e_j \mid j \in J\}$ and $\{e'_k \mid k \in K\}$ are
  the sets of closed cells of $X$ and $Y$, respectively, then
  $X \times Y$ has the weak topology with respect to the set of closed cells
  $\{ e_j \times e'_k \mid (j,k) \in J \times K\}$.
  Thus the subcategory of $\NumGen$ consisting of CW-complexes is closed
  under finite products.
\end{rem}
  
For numerically generated spaces $X$ and $Y$, let us denote
\[
Y^X = \nu\smap(X,Y). 
\]
Then there is a map $\alpha \colon Z^{X \times Y} \to (Z^Y)^X$
which assigns to $f \colon X \times Y \to Z$ the map
$\alpha(f) \colon X \to Z^Y$ given by the formula
$\alpha(f)(x)(y) = f(x,y)$ for $x \in X$ and $y \in Y$.

\begin{thm}
  \label{thm:closedness_for_NG}
  The natural map $\alpha \colon Z^{X \times Y} \to (Z^Y)^X$
  is a homeomorphism.
\end{thm}

This clearly implies the following.

\begin{crl}
  The category $\NumGen$ is a cartesian closed category.
\end{crl}

To prove the theorem, we use the relationship between $Y^X$ and
the exponentials in $\Diff$.

\begin{prp}
  For any topological spaces $X$ and $Y$, we have
  \[
  D\smap(X,Y) = \Cinfty(DX,DY).
  \]
\end{prp}

\begin{proof}
  Let $\sigma \colon U \to \smap(X,Y)$ be a map from an open subset
  $U \subset \mathbf{R}^n$.  Then $\sigma$ is a plot of $D\smap(X,Y)$
  if and only if the composite
  \[
  U \xrightarrow{\sigma} \smap(X,Y) \xrightarrow{\tau^*}
  \map(\Delta^m,Y)
  \]
  is continuous for every singular simplex $\tau \colon \Delta^m \to X$.
  But $\tau^*\sigma$ corresponds to the the composite
  \[
  U \times \Delta^m \xrightarrow{\sigma \times \tau} \smap(X,Y) \times X
  \xrightarrow{\ev} Y,
  \]
  under the homeomorphism
  \[
  \map(U \times \Delta^m,Y) \cong \map(U,\map(\Delta^m,Y)).
  \]
  Thus $\sigma$ is a plot of $D\smap(X,Y)$ if and only if
  $\ev(\sigma,\tau)$ is continuous for every $\tau$.
  which is equivalent to say that $\sigma$ is a plot of
  $\Cinfty(DX,DY)$.
\end{proof}

We are now ready to prove Theorem~\ref{thm:closedness_for_NG}.

\begin{proof}[Proof of Theorem~\ref{thm:closedness_for_NG}]
  The map $\alpha$ is a homeomorphism because it coincides with
  the composite of homeomorphisms
  \begin{align*}
    Z^{X \times Y} &= \nu\smap(X \times Y,Z) && \\
    &= T\Cinfty(D(X \times Y),DZ) && \\
    &= T\Cinfty(DX \times DY,DZ) && \\
    &\cong T\Cinfty(DX,\Cinfty(DY,DZ)) && (1) \\
    &= T\Cinfty(DX,D\smap(Y,Z)) && \\
    &= \nu\smap(X,\smap(Y,Z)) && \\
    &\cong \nu\smap(X,\nu\smap(Y,Z)) = (Z^Y)^X && (2)
  \end{align*}
  in which (1) follows from Theorem~\ref{thm:closedness_for_Diff},
  and (2) is induced by the
  numerical isomorphism $\smap(Y,Z) \to \nu\smap(Y,Z)$.
\end{proof}



\section{The space of basepoint preserving maps}
\label{sec:pointed maps}

Let $\Top_0$ and $\Diff_0$ be the categories of pointed objects
in $\Top$ and $\Diff$, respectively.
Evidently, the adjunction $(T,D)$ between $\Top$ and $\Diff$ induces
an adjunction $(T_0,D_0)$  between $\Top_0$ and $\Diff_0$.
Thus the category $\NumGen_0$ of pointed objects in $\NumGen$
can be identified with a full subcategory of $\Top_0$ consisting of
those pointed spaces $(X,x_0)$ such that $\nu{X} = X$ holds.
Clearly, $\NumGen_0$ is complete and cocomplete.

Given pointed spaces $X$ and $Y$, let $\map_0(X,Y)$ and $\smap_0(X,Y)$
denote, respectively, the subspace of $\map(X,Y)$ and $\smap(X,Y)$
consisting of basepoint preserving maps.
Then Proposition~\ref{prp:smap is map for CW-complexes} implies the following.

\begin{prp}
  \label{prp:smap_0 is map_0 for CW-complexes}
  If $X$ is a pointed CW-complex, then the inclusion map
  $\map_0(X,Y) \to \smap_0(X,Y)$ is a homeomorphism for any pointed space $Y$.
\end{prp}

As before, let us denote $Y^X = \nu\smap_0(X,Y)$.
By taking the constant map as basepoint,
$Y^X$ is regarded as an object of $\NumGen_0$.
Recall that the smash product $X \wedge Y$ of pointed spaces
$X = (X,x_0)$ and $Y = (Y,y_0)$ is defined to be the quotient of
$X \times Y$ by its subspace
$X \vee Y = X \times \{y_0\} \cup \{x_0\} \times Y$.
We now define a pointed map
\[
\alpha_0 \colon Z^{X \wedge Y} \to (Z^Y)^X
\]
to be the composite
\(
Z^{X \wedge Y} \xrightarrow{p^*} (Z,z_0)^{(X \times Y,X \vee Y)}
\xrightarrow{\alpha'} (Z^Y)^X,
\)
where the middle term $(Z,z_0)^{(X \times Y,X \vee Y)}$ denotes
the subspace of $Z^{X \times Y}$ 
consisting of those maps $f \colon X \times Y \to Z$
such that $f(X \vee Y) = \{z_0\}$ holds, $p^*$ is induced by
the natural map $p \colon X \times Y \to X \wedge Y$,
and $\alpha'$ is the restriction of the homeomorphism
\[
\nu\smap(X \times Y,Z) \cong \nu\smap(X,\nu\smap(Y,Z)).
\]
Since $p \colon X \times Y \to X \wedge Y$ is universal among
continuous maps $f \colon X \times Y \to Z$ satisfying
$f(X \vee Y) = \{z_0\}$,
the induced map $p^*$ is bijective, hence so is $\alpha_0$.

\begin{prp}
  The map $\alpha_0 \colon Z^{X \wedge Y} \to (Z^Y)^X$ is
  a homeomorphism for any $X,\ Y,\  Z \in \NumGen_0$.
\end{prp}

\begin{proof}
  Given pairs of topological spaces $(X,A)$, $(Y,B)$, we have
  \[
  (Y,B)^{(X,A)} = T\Cinfty((DX,DA),(DY,DB)),
  \]
  where $\Cinfty((DX,DA),(DY,DB))$ is the subspace of $\Cinfty(DX,DY)$
  consisting of those smooth maps $f \colon DX \to DY$ satisfying
  $f(DA) \subset DB$.
  Thus, to prove that $\alpha_0$ is a homeomorphism,
  we need only show that for any pointed diffeological spaces
  $A$, $B$ and $C$, the bijection
  \[
  p^* \colon \Cinfty((A \wedge B,*),(C,c_0)) \to
  \Cinfty((A \times B,A \vee B),(C,c_0))
  \]
  induced by the natural map
  $p \colon A \times B \to A \wedge B = A \times B/A \vee B$
  is a smooth isomorphism.
  Here $* = p(A \vee B)$ is a basepoint of $A \wedge B$, and
  $c_0$ is a basepoint of $C$.
  To see that $(p^*)^{-1}$ is smooth, let us take a plot
  \(
  \sigma \colon U \to \Cinfty((A \times B,A \wedge B),(C,c_0))
  \)
  and show that $\widetilde\sigma = (p^*)^{-1}\cdot\sigma$ is a plot
  of $\Cinfty((A \wedge B,*),(C,c_0))$.
  By definition, this is the case if for any plot
  $\tau \colon V \to A \wedge B$ the composite
  \[
  U \times V \xrightarrow{\widetilde\sigma \times \tau}
  \Cinfty((A \wedge B,*),(C,c_0)) \times (A \wedge B) \xrightarrow{\ev} C
  \]
  is a plot of $C$.
  But for any $v \in V$ there exist a neighborhood $W$ of $v$ and
  a plot $\widetilde\tau \colon W \to A \times B$ such that
  $p\widetilde\tau = \tau|W$ hold.
  Therefore, the composite $\ev(\widetilde\sigma \times \tau)$ coincides,
  on $U \times W$, with the plot
  \[
  U \times W \xrightarrow{\widetilde\sigma \times \widetilde\tau}
  \Cinfty((A \times B,A \vee B),(C,c_0)) \times (A \times B)
  \xrightarrow{\ev} C.
  \]
  This implies that $\ev(\widetilde\sigma \times \tau)$ is locally,
  hence globally, a plot of $C$.
  Thus $\widetilde\sigma = (p^*)^{-1}\cdot\sigma$ is a plot of
  $\Cinfty((A \wedge B,*),(C,c_0))$.
\end{proof}

The proposition implies a natural bijection
\[
\hom_{\NumGen_0}(X \wedge Y,Z) \cong \hom_{\NumGen_0}(X,Z^Y).
\]
Hence we have the following.

\begin{crl}
  The category $\NumGen_0$ is a symmetric monoidal closed category
  with tensor product $\wedge$ and internal hom of the form $Z^Y$.
\end{crl}

\begin{prp}
  {\rm (1)} For every pointed space $X$, the counit of the adjunction
  $\varepsilon \colon \nu{X} \to X$ is a weak homotopy equivalence.

  {\rm (2)} If $X \in \NumGen_0$, then the bijection
  \(
  \iota \colon \map_0(X,Y) \to \smap_0(X,Y)
  \)
  is a weak homotopy equivalence.
\end{prp}

\begin{proof}
  (1) Since $S^n$ is a CW-complex, we have
  \[
  \pi_n(X,x) = \pi_0\map((S^n,e),(X,x)) \cong \pi_0\smap((S^n,e),(X,x))
  \]
  for every $x \in X$.
  Therefore, the numerically continuous map $\eta \colon X \to \nu{X}$
  induces the inverse 
  to $\varepsilon_* \colon \pi_n(\nu{X},x) \to \pi_n(X,x)$.

  (2)
  We have a commutative diagram
  \[
  \xymatrix{
    \map_0(S^n,\map_0(X,Y)) \ar[r]^{\iota_*}
    & \map_0(S^n,\smap_0(X,Y)) \ar@{=}[d] \\
    & \smap_0(S^n,\smap_0(X,Y)) \\
    \map(S^n \wedge X,Y) \ar[r]^{\cong} \ar[uu]^{\alpha_0}
    & \smap(S^n \wedge X,Y) \ar[u]_{\cong},
  }
  \]
  which shows that
  \[
  \iota_* \colon \map_0(S^n,\map_0(X,Y)) \to \map_0(S^n,\smap_0(X,Y))
  \]
  is bijective.
  Thus $\smap_0(X,Y)$ has the same $n$-loops as $\map_0(X,Y)$.
  Moreover, a similar diagram as above, but $S^n$ replaced by
  $I_+ \wedge S^n$, shows that $\smap_0(X,Y)$ has
  the same homotopy classes of $n$-loops as $\map_0(X,Y)$.
  It follows that
  \[
  \iota_* \colon \pi_n(\map_0(X,Y),f) \cong \pi_n(\smap_0(X,Y),f)
  \]
  is an isomorphism for every $f \in \map_0(X,Y)$ and $n \geq 0$.
\end{proof}

\begin{crl}
  For all $X,\; Y \in \NumGen_0$, the space $Y^X$ is weakly equivalent
  to the space of maps $\map_0(X,Y)$ equipped with the compact-open topology.
\end{crl}
\section{Homology theories via enriched functors}
\label{sec:homology theories}

Let $\C_0$ be a full subcategory of $\NumGen_0$.
Then $\C_0$ is an enriched category over $\NumGen_0$ with hom-objects
$F_0(X,Y) = Y^X$.
A covariant functor $T$ from $\C_0$ to $\NumGen_0$ is called enriched
if the correspondence $F_0(X,Y) \to F_0(TX,TY)$,
which maps $f$ to $Tf$, is a morphism in $\NumGen_0$,
that is, a basepoint preserving continuous map. 
Similarly, a contravariant functor $T$ form $\C_0$ to $\NumGen_0$
is called enriched if the map $F_0(X,Y) \to F_0(TY,TX)$
is a morphisms in $\NumGen_0$.  

Let $T \colon \C_0 \to \NumGen_0$ be an enriched functor.
Then for any pointed map $f \colon X \wedge Y \to Z$ such that
$Y$ and $Z$ are objects of $\C_0$ the composite
\[
X \xrightarrow{\alpha_0(f)} F_0(Y,Z) \xrightarrow{T} F_0(TY,TZ)
\]
induces, by adjunction, a pointed map $X \wedge TY \to TZ$.
Similarly, an enriched cofunctor $T \colon \C_0^{\op} \to \NumGen_0$
assigns $X \wedge TZ \to TY$ as an adjunct to the composite
$X \to F_0(Y,Z) \to F_0(TZ,TY)$.

\begin{prp}
  \label{prp:continuity_means_homotopy_invariance}
  Enriched functors and cofunctors from $\C_0$ to $\NumGen_0$
  preserve homotopies.
\end{prp}

\begin{proof}
  Let $h \colon I_+ \wedge X \to Y$ be a pointed homotopy between
  $h_0$ and $h_1$.
  Then an enriched functor $T \colon \C_0 \to \NumGen_0$ induces
  a homotopy $I_+ \wedge TX  \to TY$ between $Th_0$ and $Th_1$.
  Similarly, an enriched cofunctor $T$ induces a homotopy
  $I_+ \wedge TY \to TX$ between $Th_0$ and $Th_1$.
\end{proof}

\begin{crl}
  If $T \colon \C_0 \to \NumGen_0$ is an enriched functor then
  a homotopy equivalence $f \colon X \to Y$ induces isomorphisms
  $Tf_* \colon \pi_n TX \cong \pi_n TY$ for $n \geq 0$.
  Similarly, if $T$ is an enriched cofunctor then $f$ induces isomorphisms
  $Tf_* \colon \pi_n TY \cong \pi_n TX$ for $n \geq 0$.
\end{crl}

From now on, we assume that $\C_0$ satisfies the following conditions:
(i) $\C_0$ contains all finite CW-complexes.
(ii) $\C_0$ is closed under finite wedge sum. 
(iii) If $A \subset X$ is an inclusion of objects in $\C_0$ then
its cofiber $X \cup CA$ belongs to $\C_0$;
in particular, $\C_0$ is closed under the suspension functor
$X \mapsto \Sigma X$.
The category $\FCW_0$ of finite CW-complexes is a typical example of
such a category.

Given a continuous map $f \colon X \to Y$,
let
\[
E(f) = \{(x,l) \in X \times \map(I,Y) \mid f(x) = l(0)\}
\]
be the mapping track of $f$.
Then the map $p \colon E(f) \to Y$, $p(x,l) = l(1)$, has
the homotopy lifting property for all spaces, and hence induces a bijection
$p_* \colon \pi_{n+1}(E(f),F(f)) \to \pi_{n+1}Y$ for all $n \geq 0$,
where $F(f)$ denotes the fiber of $p$ at the basepoint of $Y$.
A sequence of pointed maps $Z \xrightarrow{i} X \xrightarrow{f} Y$ is called
a homotopy fibration sequence if there is a homotopy of pointed maps from
$f \circ i$ to the constant map such that the induced map $Z \to F(f)$
is a weak homotopy equivalence.

\begin{dfn}
  An enriched functor $T \colon \C_0 \to \NumGen_0$ is called linear
  if for every pair of objects $(X,A)$ with $A \subset X$, the sequence
  \[
  TA \to TX \to T(X \cup CA),
  \]
  induced by the cofibration sequence $A \subset X \subset X \cup CA$,
  is a homotopy fibration sequence with respect to the null homotopy
  of $TA \to T(X \cup CA)$ coming from the contraction of $A$ within
  the reduced cone $CA$.
  Likewise, an enriched cofunctor $T \colon \C_0^{\op} \to \NumGen_0$
  is called linear if the induced sequence
  \[
  T(X \cup CA) \to TX \to TA
  \]
  is a homotopy fibration sequence.
\end{dfn}

If $T$ is a linear functor then every pair $(X,A)$ gives rise to
an exact sequence of pointed sets
\[
\cdots \to
\pi_{n+1} T(X \cup CA) \xrightarrow{\varDelta} \pi_n TA
\xrightarrow{Ti_*} \pi_n TX \xrightarrow{Tj_*} \pi_n T(X \cup CA)
\to \cdots
\]
terminated at $\pi_0 T(X \cup CA)$.
Here $i$ and $j$ denote the inclusions $A \subset X$ and
$X \subset X \cup CA$, respectively, and
$\varDelta$ is the composite
\[
\pi_{n+1}T(X \cup CA) \xrightarrow{p_*^{-1}} \pi_{n+1}(E(Tj),F(Tj))
\xrightarrow{\partial} \pi_n F(Tj) \xrightarrow{\nu_*^{-1}} \pi_n TA.
\]
Similarly, a linear cofunctor $T$ induces an exact sequence 
\[
\cdots \to
\pi_{n+1} TA \xrightarrow{\varDelta} \pi_n T(X \cup CA)
\xrightarrow{Tj_*} \pi_n TX \xrightarrow{Ti_*} \pi_n TA
\to \cdots
\]
terminated at $\pi_0 TA$.

\begin{thm}
  \label{thm:homology_theory}
  For every linear functor $T \colon \C_0 \to \NumGen_0$,
  there exists a generalized homology theory $X \mapsto \{ h_n(X;T) \}$
  defined on $\C_0$ such that $h_n(X;T)$ is isomorphic to
  $\pi_n TX$ if $n \geq 0$, and to $\pi_0 T(\Sigma^{-n}X)$ otherwise.
\end{thm}

\begin{proof}
  We first show that
  the map $T(X \vee Y) \to TX \times TY$
  induced by the projections of $X \vee Y$ onto $X$ and $Y$
  is a weak equivalence.
  This means that the functor $\Gamma \to \NumGen_0$,
  which maps a pointed finite set $\mathbf{k} = \{0,1,\dots,k\}$ to
  $T(X \wedge \mathbf{k}) = T(X \vee \cdots \vee X)$, is special
  in the sense that the natural map $T(X \wedge \mathbf{k}) \to T(X)^k$
  is a weak equivalence for all $k \geq 0$.
  Hence $\pi_n TX$ is an abelian monoid
  with respect to the multiplication
  \[
  \pi_n TX \times \pi_n TX \cong \pi_n T(X \vee X)
  \xrightarrow{T\nabla_*} \pi_n TX
  \]
  induced by the folding map $\nabla \colon X \vee X \to X$.
  This multiplication coincides with the standard
  multiplication of $\pi_n TX$ because they are compatible with each other.
  In particular, $\pi_1 TX$ is an abelian group for all $X$.
  Moreover, any pointed map $f \colon X \to Y$ induces a natural
  transformation $T(X \wedge \mathbf{k}) \to T(Y \wedge \mathbf{k})$,
  hence a homomorphism of abelian monoids
  $Tf_* \colon \pi_n TX \to \pi_n TY$ for all $n \geq 0$.

  To see that $T(X \vee Y) \to TX \times TY$ is a weak equivalence,
  it suffices to show that the sequence $TX \to T(X \vee Y) \to TY$,
  induced by the inclusion $X \to X \vee Y$ and the projection
  $X \vee Y \to Y$, is a homotopy fibration sequence.
  But $TX \to T(X \vee Y) \to TY$ is homotopy equivalent to
  the homotopy fibration sequence $TX \to T(X \vee Y) \to T(CX \vee Y)$
  through the homotopy equivalence $T(CX \vee Y) \simeq TY$
  induced by the retraction $CX \vee Y \to Y$.

  Next consider the homotopy exact sequence associated with the sequence
  $TX \to T(CX) \to T(\Sigma X)$.
  As $T(CX)$ is weakly contractible, we obtain for every $n \geq 0$
  a short exact sequence of abelian monoids
  \[
  0 \to \pi_{n+1} T(\Sigma X) \xrightarrow{\varDelta} \pi_n TX \to 0.
  \]
  Since $\pi_{n+1} T(\Sigma X)$ is an abelian group,
  the homomorphism $\varDelta$ is injective, hence an isomorphism.
  But this in turn means that $\pi_n TX$ is an abelian group for $n \geq 0$.
  Therefore, $h_n(X; T)$ is an abelian group for all $n \in \mathbf{Z}$.

  For a pointed map $f \colon X \to Y$, we define
  \[
  h_n(f) \colon h_n(X;T) \to h_n(Y;T)
  \]
  to be the homomorphism induced by $Tf \colon TX \to TY$
  for $n \geq 0$, and $T(\Sigma^{-n}f)$ for $n < 0$.
  It is easy to see that the functor $X \mapsto \{h_n(X;T)\}$
  together with a natural isomorphism $h_n(X;T) \cong h_{n+1}(\Sigma X;T)$,
  given by $\varDelta^{-1}$ for $n \geq 0$ and the identity for $n < 0$,
  satisfies the homotopy and exactness axioms.
\end{proof}

\begin{exa}
  \label{example:homology theories}
  (1) Let $AG(X)$ denote the topological free abelian group generated by
  $X$ with basepoint identified with $0$.
  Then the correspondence $X \mapsto AG(X)$ defines a linear enriched
  functor $\NumGen_0 \to \NumGen_0$.
  By the Dold-Thom theorem, $h_n(X,AG)$ is isomorphic to the singular
  homology group for a CW-complex $X$.
  But this is not the case for general $X$.
  In fact, the singular homology theory cannot be represented by
  a linear enriched functor $T \colon \NumGen_0 \to \NumGen_0$,
  because the singular homology does not satisfy the wedge axiom
  (cf.\ \cite[6.11]{dold-thom}), contradicting to the fact that
  the homology theory given by a linear enriched functor must satisfy
  the axiom.

  (2) Labeled configuration spaces also give rise to linear enriched functors.
  Let $C^M(\mathbf{R}^{\infty})$ be the configuration space of
  finite points in $\mathbf{R}^{\infty}$ with labels in
  a partial abelian monoid $M$.
  Then $\Omega C^{\Sigma X \wedge M}(\mathbf{R}^{\infty})$ is
  a linear enriched functor of $X$,
  and hence defines a connective homology theory.  
  We have shown in \cite{config} that the classical homology,
  the stable homotopy, and the connective $K$-homology theories
  can be obtained in this way.
\end{exa}

If $T \colon \C_0 \to \NumGen_0$ is an enriched functor then
for every $X \in \C_0$ the spaces $T(\Sigma^n X)$ together with
the maps $S^1 \wedge T(\Sigma^n X) \to T(\Sigma^{n+1} X)$
induced by the homeomorphism $S^1 \wedge \Sigma^n X \cong \Sigma^{n+1} X$
form a prespectrum $\spec{TX}$. 
Following Goodwillie \cite{goodwillie_I}, we call $\spec{TX}$
the derivative of $T$ at $X$.
Let $L(\spec{TX})$ 
be the spectrification of $\spec{TX}$ (cf.\ \cite{lewis-may-steinberger}).
%
%
Then its zeroth space $L(\spec{TX})_0$ is an infinite loop space
and the correspondence $X \mapsto L(\spec{TX})_0$ defines
an enriched functor $LT \colon \C_0 \to \NumGen_0$.
If, moreover, $T$ is linear then Theorem~\ref{thm:homology_theory}
implies that the natural map $TX \to LTX$ is a weak equivalence
for every $X$; hence $LT$ defines the same homology theory as $T$.

An enriched functor $T$ is called stable if the natural map
$TX \to LTX$ is a homeomorphism for every $X$.
In particular, $LT$ is stable for any $T$.
Let $\SLEF$ be the category of stable linear enriched functors
$\C_0 \to \NumGen_0$ with enriched natural transformations as morphisms.
Then there is a functor $D$ from $\SLEF$ to the category $\Spec$ of spectra
which maps a stable functor $T$ to its derivative $\spec{TS^0}$ at $S^0$.
We have shown that the homology theory $h_{\bullet}(-; T)$ induced by
a linear enriched functor $T$ is represented by the spectrum
$D(LT) = L(\spec{TS^0})$.

Conversely, any homology theory represented by a spectrum
comes from a linear enriched functor.
In fact, $D$ has a left adjoint $I \colon \Spec \to \SLEF$
defined as follows:
For a spectrum $E = \{E_n\}$, $I{E}$ is the enriched functor which
maps $X$ to the zeroth space $L(E \wedge X)_0$ of the spectrification
of the prespectrum $E \wedge X = \{E_n \wedge X\}$.
The unit $E \to DIE$ and the counit $IDT \to T$ of the adjunction are
weak equivalences given by the maps
\begin{eqnarray*}
  &E_n \to \Omega^n(E_n \wedge S^n) \to \Omega^{\infty}(E_{\infty} \wedge S^n)
  = L(E \wedge S^n)_0 = I{E}(S^n) = DI{E}_n,&
  \\
  &ID{T}X = L(\spec{TS^0} \wedge X)_0 \xrightarrow{L\mu}
  L(\spec{TX})_0 = LTX \cong TX,&
\end{eqnarray*}
where $\mu \colon \spec{TS^0} \wedge X \to \spec{TX}$ is a map of
prespectra consisting of the maps $T(S^n) \wedge X \to T(\Sigma^n X)$
induced by the identity of $S^n \wedge X$.

Let us regard $\Spec$ and $\SLEF$ as a model category with respect to
the classes of weak equivalences, fibrations, and cofibrations
consist, respectively, of level weak equivalences, level fibrations,
and morphisms that have the left lifting property with respect to
the class of trivial fibrations.
Then we have the following.

\begin{prp}
  The functor $D \colon \SLEF \to \Spec$ is a right Quillen equivalence,
  and hence induces an equivalence between the homotopy categories.
\end{prp}

\begin{crl}
  The homotopy category of $\SLEF$ is equivalent to the stable category.
\end{crl}

\begin{rem}
  Let $T \colon \C_0 \to \NumGen_0$ be an enriched functor.
  Then for every real inner product space $V$ the standard orthogonal
  group action on $V$ induces an orthogonal group action on the one-point
  compactification $S^V$ of $V$, and hence on $T(\Sigma^V X)$ where
  $\Sigma^V X = S^V \wedge X$.  
  With respect to the orthogonal group action on each $T(\Sigma^V X)$,
  the prespectrum $\{ T(\Sigma^V X) \}$ indexed by real inner product
  spaces $V$ form an orthogonal prespectrum.
  Thus the derivative $\spec{TX} = \{ T(\Sigma^n X) \}$ at $X$
  extends to an orthogonal spectrum, and consequently
  is a symmetric spectrum with respect to the symmetric group action
  on $T(\Sigma^n X)$ induced by the embedding of the symmetric group
  into the orthogonal group as the subgroup of coordinate permutations.
\end{rem}
\section{Bivariant homology-cohomology theories}
\label{sec:homology and cohomology}

We now introduce the notion of a bilinear functor,
and describe a passage from bilinear functors
to generalized cohomology theories.
In fact, we shall show that a bilinear functor gives rise to
a pair of generalized homology and cohomology theories,
or in other words, a bivariant homology-cohomology theory.

Let $F \colon {\C_0}^{\op} \times \C_0 \to \NumGen_0$ be a bivariant
functor which is contravariant with respect to the first argument,
and is covariant with respect to the second argument.
We say that $F$ is enriched (over $\NumGen_0$)
if for all pointed spaces $X$, $X'$, $Y$, and $Y'$, the map
\[
F_0(X',X) \times F_0(Y,Y') \to F_0(F(X,Y),F(X',Y')),\quad
(f,g) \mapsto F(f,g)
\]
is continuous and is pointed in the sense that
if either $f$ or $g$ is constant then so is $F(f,g)$.

\begin{dfn}
  An enriched bifunctor $F \colon {\C_0}^{\op} \times \C_0\to \NumGen_0$
  is called a bilinear functor if for all $(X,A)$ and $(Y,B)$
  the sequences
  \[
  \begin{split}
    & F(X \cup CA,Y) \to F(X,Y) \to F(A,Y),
    \\
    & F(X,B) \to F(X,Y) \to F(X,Y \cup CB)
  \end{split}
  \]
  are homotopy fibration sequences.
\end{dfn}

\begin{exa}
  Given a linear functor $T \colon \C_0 \to \NumGen_0$, there is
  a bilinear functor $F_T$ given by the formula $F_T(X,Y) = F_0(X,TY)$.
  Such a bilinear functor is said to be of the normal form.
  Conversely, a bilinear functor $F$ determines a linear functor
  $T \colon \C_0 \to \NumGen_0$ by putting $TX = F(S^0,X)$,
  called the covariant part of $F$.
  There is a natural transformation $F \to F_T$ such that
  \(
  F(X,Y) \to F_T(X,Y) = F_0(X,F(S^0,Y))
  \)
  maps every $\xi \in F(X,Y)$ to the map
  $X \to F(S^0,Y),\ x \mapsto F(x,1_Y)(\xi)$.
  Here we identify $x \in X$ with the evident map $S^0 \to X$.
\end{exa}

\begin{thm}
  \label{thm:bivariant_theory}
  For every bilinear functor
  $F \colon {\C_0}^{\op} \times \C_0 \to \NumGen_0$,
  there exist a generalized homology theory $X \mapsto \{h_n(X;F)\}$ and
  a generalized cohomology theory $X \mapsto \{h^n(X;F)\}$ such that
  \begin{equation}
    \label{eq:homology_cohomology_groups}
    h_n(X;F) \cong \pi_0 F(S^{n+k},\Sigma^k X), \quad
    h^n(X;F) \cong \pi_0 F(\Sigma^k X,S^{n+k})
  \end{equation}
  hold whenever $k,\, n+k \geq 0$.
  Moreover, $h_n(X;F)$ is naturally isomorphic to the $n$-th homology group
  $h_n(X;T)$ given by the covariant part $T$ of $F$.
\end{thm}

\begin{proof}
  Since $F(X,Y)$ is linear with respect to $Y$,
  $\pi_n F(X,Y)$ is an abelian group for all $X$, $Y$ and $n \geq 0$.
  Clearly, this abelian group structure is natural with respect to
  both $X$ and $Y$.
  Moreover, the bilinearity of $F$ implies natural isomorphisms
  \begin{equation*}
    \pi_n F(X,Y) \cong \pi_{n+1} F(X,\Sigma Y), \quad
    \pi_n F(\Sigma X,Y) \cong \pi_{n+1} F(X,Y)
  \end{equation*}
  Consequently, there is a natural isomorphism
  $\pi_0 F(X,Y) \cong \pi_0 F(\Sigma X,\Sigma Y)$,
  called the suspension isomorphism.

  For every pointed space $X$ and every integer $n$, let us define
  \[
  h_n(X;F) = \colim_{k \to \infty} \pi_0 F(S^{n+k},\Sigma^k X), \
  h^n(X;F) = \colim_{k \to \infty} \pi_0 F(\Sigma^k X,S^{n+k})
  \]
  where the colimits are taken with respect to the suspension isomorphisms.
  Clearly (\ref{eq:homology_cohomology_groups}) holds, and we have
  $h_n(X;F) \cong h_n(X;T)$ where $T$ is the covariant part of $F$.
  Thus the functor $X \mapsto \{h_n(X;F)\}$ together with
  the evident natural isomorphism $h_n(X;F) \cong h_{n+1}(\Sigma X;F)$
  defines a generalized homology theory.
  Similarly, the covariant functor $X \mapsto \{h^n(X;F)\}$
  together with the natural isomorphism $h^{n+1}(\Sigma X;F) \cong h^n(X;F)$
  defines a generalized cohomology theory,
  because it satisfies the homotopy and exactness axioms.
\end{proof}
  
\begin{exa}

  (1) Let $F \colon {\NumGen_0}^{\op} \times \NumGen_0 \to \NumGen_0$
  be the bilinear functor given by the formula $F(X,Y) = AG(Y)^X$.
  Then $h^n(X,F)$ is the singular cohomology group for a CW-complex $X$.
  But there exists no bilinear functor representing
  the singular cohomology groups of all $X$.
  (Cf.\ Example~\ref{example:homology theories} (2).)

  (2) The second author constructs in \cite{yoshida} a bilinear functor
  $F$ such that $h^n(X; F)$ is isomorphic to the \v{C}ech cohomology group
  for all $X$,
  and $h_n(X;F)$ is isomorphic to the Steenrod homology group
  for any compact metrizable space $X$.

\end{exa}

\begin{thebibliography}{1}

\bibitem{dold-thom}
A.~Dold and R.~Thom.
\newblock Quasifaserungen und unendliche symmetrische produkte.
\newblock {\em Ann. Math.}, 67:239--281, 1958.

\bibitem{goodwillie_I}
T.~Goodwillie.
\newblock Calculus i: The first derivative of pseudoisotopy theory.
\newblock {\em {$K$}-{T}heory}, 4:1--27, 1990.

\bibitem{zemmour:diffeology}
P.~Iglesias-Zemmour.
\newblock Diffeology.
\newblock (working document).

\bibitem{lewis-may-steinberger}
G.~Lewis, J.~P. May, and M.~Steinberger.
\newblock {\em Equivariant Stable Homotopy Theory}, volume 1213 of {\em Lecture
  Notes in Math.}
\newblock Springer-Verlag, 1986.

\bibitem{config}
K.~Shimakawa.
\newblock Configuration spaces with partially summable labels and homology
  theories.
\newblock {\em Math. J. Okayama Univ.}, 43:43--72, 2001.

\bibitem{yoshida}
K.~Yoshida.
\newblock A continuous bifunctor representing \v{C}ech cohomology.
\newblock preprint.

\end{thebibliography}

\end{document}